\documentclass{article}

\usepackage{amsmath}
\usepackage{amsfonts}
\usepackage{amssymb}
\usepackage{amsthm}
\usepackage{enumerate}
\usepackage{bbm}
\usepackage{hyperref}

\newtheorem{theorem}{Theorem}[section]
\newtheorem{lemma}[theorem]{Lemma}
\newtheorem{claim}[theorem]{Claim}

\newtheorem{corollary}[theorem]{Corollary}

\newtheorem{proposition}[theorem]{Proposition}

\newtheorem{conjecture}[theorem]{Conjecture}

\newcommand{\suchthat}{\;\ifnum\currentgrouptype=16 \middle\fi|\;}

\newcommand{\R}{\mathbb{R}}

\newcommand*{\defeq}{\stackrel{\text{def}}{=}}

\DeclareMathOperator{\disc}{disc}
\DeclareMathOperator{\ent}{H}
\newcommand{\E}{{\rm I\kern-.3em E}}
\newcommand{\Var}{\mathrm{Var}}

\begin{document}
\title{Discrepancy in random hypergraph models }
\author{
Aditya Potukuchi \thanks{Department of Computer Science,  Rutgers University. {\tt aditya.potukuchi@cs.rutgers.edu}.} 
}

\maketitle

\begin{abstract} 
We study hypergraph discrepancy in two closely related random models of hypergraphs on $n$ vertices and $m$ hyperedges. The first model, $\mathcal{H}_1$, is when every vertex is present in exactly $t$ randomly chosen hyperedges. The premise of this is closely tied to, and motivated by the Beck-Fiala conjecture. The second, perhaps more natural model, $\mathcal{H}_2$, is when the entries of the $m \times n$ incidence matrix is sampled in an i.i.d. fashion, each with probability $p$. We prove the following:

\begin{itemize}
\item In $\mathcal{H}_1$, when $\log^{10}n \ll t \ll \sqrt{n}$, and $m = n$, we show that the discrepancy of the hypergraph is almost surely at most $O(\sqrt{t})$. This improves upon a result of Ezra and Lovett for this range of parameters.
\item In $\mathcal{H}_2$, when $p= \frac{1}{2}$, and $n = \Omega(m \log m)$, we show that the discrepancy is almost surely at most $1$. This answers an open problem of Hoberg and Rothvoss.
\end{itemize}
\end{abstract}

\section{Introduction}

Let $\mathcal{H} = (V,E)$ be a hypergraph, with $V$ as the set of vertices, and $E \subseteq 2^{V}$ as the set of (hyper)edges. Let $\mathcal{X} = \{\chi : V \rightarrow \{\pm 1\}\}$, be the set of $\pm 1$ colorings of $V$, and for $\chi \in \mathcal{X}$, and $e \in E$, denote $\chi(e) := \sum_{v \in e}\chi(v)$. The discrepancy of $\mathcal{H}$, denoted by $\disc(\mathcal{H})$ is defined as:

\[
\disc(\mathcal{H}) \defeq \min_{\chi \in \mathcal{X}}\max_{e \in E} |\chi(e)|.
\]

The study of this quantity, which seems to have been first defined in a paper of Beck~\cite{BEC81}, has led to some very interesting results with very diverse techniques. One of the most interesting open problems in discrepancy theory is what is known as the Beck-Fiala conjecture, regarding the discrepancy of $t$-regular hypergraphs. We say a hypergraph is $t$-regular if every vertex occurs in exactly $t$ edges.

\begin{conjecture}[Beck-Fiala conjecture]
For a $t$-regular hypergraph $\mathcal{H}$, we have that:
\[
\disc(\mathcal{H}) = O(\sqrt{t})
\]
\end{conjecture}

Although this conjecture is usually stated for bounded degree hypergraphs (instead of regular ones), this is not really an issue. One can always add hyperedges containing just a single vertex, and make it regular, and increasing the discrepancy by at most $1$. Beck and Fiala~\cite{BF81} also proved that for any $t$-regular hypergraph $\mathcal{H}$,

\[
\disc(\mathcal{H}) \leq 2t-1
\] 

This is more commonly known as the Beck-Fiala theorem. Essentially the same proof can be done a bit more carefully to get a bound of $2t - 3$. Given the conjecture, it is perhaps surprising that the best upper bound, due to Bukh~\cite{BUK16}, is ``stuck at" $2t - \log^*t$ for large enough $t$.

It is possible that one of the reasons that the discrepancy upper bounds are so far away from the conjectured bound (assuming it's true), could be our inability to handle many `large' hyperedges. Indeed, if one is offered the restriction that each hyperedge is also of size $O(t)$ (regular and `almost uniform'), then a folklore argument using the Lov\'{a}sz Local Lemma shows that the discrepancy is bounded by $O(\sqrt{t \log t})$. 

All the results above do not really say much about the discrepancy of \emph{general} $n$-vertex, $n$-edge hypergraphs, since we can only bound the regularity and uniformity by $n$. For example, a randomly chosen coloring will give discrepancy $O(\sqrt{n \log n})$ with high probability. An extremely important result in this direction, that is also related to this work, is due to Spencer~\cite{SPE85} which states that for any hypergraph $\mathcal{H} =(V,E)$ where $|V| = |E| = n$, we have $\disc(\mathcal{H}) \leq 6 \sqrt{n}$. This approach uses the `partial coloring lemma' of Beck~\cite{BEC81}, which forms the basis for many other results, all of which seek to assign a coloring to the vertices in a sequence of steps, coloring a constant fraction of uncolored vertices in each step. 

One must note that the proof of Spencer was not algorithmic. In fact, Alon and Spencer(\cite{AS00}, $\S 14..5$) suggested that an efficient algorithm to output a coloring that achieves $O(\sqrt{n})$ discrepancy is not possible. However, this was shown to be incorrect by Bansal~\cite{BAN10} who shows an efficient algorithm to do the same task. However, the analysis of this algorithm still relied on the (non-algorithmic) discrepancy bound of $6 \sqrt{n}$. Later, Lovett and Meka~\cite{LM15} gave a `truly constructive' proof of the fact that the discrepancy is $O(\sqrt{n})$ which was also very interesting for the techniques used. More recently, a result due to Rothvoss~\cite{ROT17} gives an alternative proof of the same bound, which is also constructive, and more general.

\subsection{Discrepancy in random settings}

Recently, there has been some interest in the discrepancy of \emph{random} hypergraphs. Motivated by the seeming difficulty of bounding the discrepancy of general $t$-regular hypergrpahs, Ezra and Lovett~\cite{EL15} initiated the study discrepancy of \emph{random} $t$-regular hypergraphs. By random $t$-regular hypergraph, we mean the hypergraph sampled by the following procedure: We fix $n$ vertices $V$ and $m$ (initially empty) hypergedges $E$. Each vertex in $V$ chooses $t$ (distinct) hyperedges in $E$ uniformly and independently to be a part of. They showed that if $m \geq n$, then the discrepancy of such a hypergraph is almost surely $\sqrt{t \log t}$ for any growing function $t = t(n)$. The proof idea is the following: First observe that most of the hyperedges have size $O(t)$. For the remaining large edges, one can delete one vertex from every hyperedge and make them pairwise disjoint. This allows us to apply the folklore Lov\'{a}sz Local Lemma based argument, but with a slight modification which makes sure that the large edges have discrepancy at most $2$. Our first result states that in the case where $t$ is large enough, one gets the bound of $O(\sqrt{t})$, and is the content of Section~\ref{sec:H1}. More formally, we prove the following:

\begin{theorem}
\label{thm:H1}
There are constants $C_1$ and $C_2$ such that the following statement holds:
Let $\mathcal{H}_1$ be a random $t$-regular hypergraph on $n$ vertices and $n$ hyperedges, where $ C_1\log^{10}n \ll t \ll \sqrt{n} $. Then,

\[
\Pr\left(\disc(\mathcal{H}_1) \leq C_2\sqrt{t}\right) \geq 1 - o(1)
\]

Moreover, such a coloring can be outputted algorithmically in randomized $\operatorname{poly}(n)$ time.
\end{theorem}

As mentioned previously, this is proved via. the partial coloring approach, which is inspired by a later paper of Spencer~\cite{SPE88}, which computes the discrepancy of the projective plane upto a constant factor. A more general bound was also obtained by Matou\v{s}ek~\cite{MAT95}, who upper bounds the discrepancy of set systems of bounded VC-dimension. However, we will prove a constructive version which is essentially to apply the partial coloring lemma of Lovett and Meka. One ingredient that is needed to analyze this coloring is the following:

\begin{theorem}
\label{thm:norm}
 Let $M$ be the incidence matrix of a random $t$-regular set system: For any vector $v \in \mathbf{1}^{\perp}$, we have $\|Mv\| = O(\sqrt{t}\|v\|)$. 
\end{theorem}

Although many variations of the above theorem are known and standard, one must need to verify it for our setting too. It should come as no surprise that, the proof follows that of Kahn and Szemer\'{e}di's in~\cite{FKS}, and is the content of Section~\ref{sec:norm}.

Another result that Ezra and Lovett proved, in~\cite{EL15}, that if $n \gg m^t$, then the discrepancy of the random $t$-regular hypergraph is at most $O(1)$. This has recently been improved due to Franks and Saks~\cite{FS18}, who prove that as long as $n = \Omega\left(\frac{m^3}{\log^2 m}\right)$, the discrepancy of the random $t$-regular hypergraph is almost surely at most $1$. 

Another closely related model of random hypergraphs on $n$ vertices and $m$ edges is to choose every vertex-edge incidence independently with probability $p$. When $p$ in an appropriate range, this behaves `similar' to the previous random $pn$-regular model. In this model, recently Hoberg and Rothvoss~\cite{HR18} showed that as long as $pm = \Omega(\log n)$, and $n = \Omega(m^2 \log m)$, one has that $\disc(\mathcal{H}) \leq 1$. Of course, it is easy to see that to have discrepancy constant, one must have $n = \Omega(m \log m)$. Motivated by this seeming gap, they ask a bound of $o(\sqrt{m})$ can be achieved when $n$ is smaller, i.e., when $m = \Theta\left( \frac{n}{\log n}\right)$ and $p = \frac{1}{2}$. Our second result answers this in the affirmative.

\begin{theorem}
\label{thm:H2}
There is a constant $C_3$ such that the following holds:
Let $\mathcal{H}_2$ be a random hypergraph on $n$ vertices and $m$ edges where $m = \frac{n}{C_3\log n}$ where for every vertex $v$ and every edge $e$, $\mathbbm{1}[v \in e]$ is independent with probability $\frac{1}{2}$. Then, we have:

\[
\Pr(\disc(\mathcal{H}_2) \leq 1) \geq 1 - o(1)
\]
\end{theorem}

The $o(1)$ in the above theorem is $O\left(\frac{1}{\sqrt{\log n}}\right)$. One thing that is common to~\cite{FS18} and~\cite{HR18} is a multivariate \emph{local limit theorem} on the number of colorings that give discrepancy at most $1$, inspired by~\cite{KLP12}. Our approach for this is more direct and essentially just amounts to a careful second moment computation of the number of colorings which give discrepancy at most $1$, and is the content of Section~\ref{section:H2}.

\subsection{Random vs. Pseudorandom setting}

We would like to (informally) point out some difference in the flavor of Theorems~\ref{thm:H1} and~\ref{thm:H2} One way of proving that some property $A$ occurs almost surely for some random object $X$, is the following strategy:

\begin{itemize}
\item[1.] Show that a certain property $B$ occurs almost surely for $X$.
\item[2.] Show that the occurrence of $B$ deterministically guarantees the occurrence of $A$.
\end{itemize}

Of course, this is trivial if we think of $B$ as the same as $A$. However, one could think of $B$ as an easy to describe property that also implies several other events that occur almost surely. For example, if our object of interest $X$ was a random $d$-regular graph, property $B$ could be that the second largest eigenvalue is at most $O(\sqrt{d})$. This property alone \emph{deterministically} implies some other properties that we are often interested in. 

In our setting of Theorem~\ref{thm:H1}, the following properties alone \emph{deterministically} imply that the discrepancy of the hypergraph is $O(\sqrt{t})$.

\begin{itemize}
\item[1.] Every edge $e \in E$ is such that $|e| = O(t)$.
\item[2.] Let $M$ be the incidence matrix of the hypergraph. For every vector $v \in \mathbf{1}^{\perp}$, we have that $\|Mv\| = O(\sqrt{t} \|v\|)$.
\end{itemize}

On the other hand, Theorem~\ref{thm:H2} does not enjoy any such property, and we do not know if there is any simple `pseudorandomness' property that guarantees that the discrepancy is at most $1$.

\paragraph{Related work} Bansal and Meka (private communication) have informed us of independent work, where they proved a theorem similar to Theorem~\ref{thm:H1} but with a wider range of parameters, namely, $t = \Omega(\log \log m)$ suffices for all $m$ and $n$ for a bound of $O(\sqrt{t})$.

\section{Notation and preliminaries}

\subsection{General}

Hypergraphs are denoted by a tuple $(V,E)$, where $V$ is the set of vertices and $E \subseteq 2^V$ is the set of (hyper)-edges. We will often use the term `hypergraph' interchangeably with `set-system', where the sets are just the hyperedges.

\subsubsection{Incidence matrix}

For a hypergraph (set-system) $\mathcal{H}$ with vertices (elements) $V$ and hyperedges (sets) $E$, it is natural to look at the $m \times n$ matrix $M = M(\mathcal{H})$ where rows are indexed by the $E$ and the columns by $V$, and the entries are $M_{v,e} = \mathbbm{1}[v \in e]$. This matrix is commonly known as the \emph{incidence matrix} of the hypergraph. The discrepancy of the hypergraph is defined in matrix-terminology as:

\[
\disc(\mathcal{H}) \defeq \min_{u \in \{\pm 1\}^n} \|Mu\|_{\infty}
\]

We will primarily use this notation, especially in Section~\ref{section:H2}. 

\subsubsection{High probability events}

A sequence of events $\{\mathcal{E}_n\}$ is said to hold \emph{almost surely} if $\operatorname{lim}_{n \rightarrow \infty } \Pr(\mathcal{E}_n) = 1$. Throughout, we shall say ``event" and ``random variable" to means a sequence of events and random variables respectively, which depend on a hidden parameter $n$. In this case, $n$ is always the number of vertices in the hypergraph.

\subsection{Concentration inequalities}

We will use a few different concentration inequalities in the proofs. The first and simplest one is Chebyshev's inequality which states that for a random variable $X$ with mean $\mu$ with variance $\sigma^2$, we have:

\begin{equation*}
\label{eqn:chebyshev}
\Pr(|X - \mu| \geq t\mu) \leq \frac{\sigma^2}{\mu^2t^2}
\end{equation*}

We will actually use an immediate consequence of this:

\begin{corollary}
\label{corr:main}
If $X$ is a non-negative random variable, and $\epsilon$ is a positive real number such that $\Var(X) \leq \epsilon \E[X]^2$, then we have
\[
\Pr(X \neq 0) \geq 1 - \epsilon.
\]
\end{corollary}

We will also need the Chernoff bound. When the random variable $X$ with mean $\mu$ is a sum of independent $n$ independent random variables, distributed in $[0,1]$, we have:

\begin{equation}
\label{ineq:Chernoff}
\Pr(X \geq \mu + t) \leq 2e^{-\frac{2t^2}{n}}
\end{equation}

Let $G = G(N,M,K)$ denote the hypergeometric random variable with parameters $N,M,K$, i.e., the intersection size $|S \cap [K]|$ for a randomly chosen $S \subseteq [N]$ with $|S| = M \leq N$. The same large deviation inequality also holds for hypergeometric variables. Denote $\mu := \E[G(N,M,K)] = MK/N$ We have 

\begin{equation}
\label{ineq:Hypergeom}
\Pr(G \geq \mu + t) \leq 2e^{- \frac{2t^2}{M}}
\end{equation}

\subsubsection{Martingales and the method of bounded variances}

The most general concentration inequality (this subsumes~(\ref{ineq:Chernoff}) and~(\ref{ineq:Hypergeom})) that will be used here is a martingale inequality. A sequence of random variables $X_0,X_1,\ldots, X_n$ martingale with respect to another sequence of random variables $Z_0, Z_1,\ldots, Z_n$ such that for all $i \in [n-1]$, we have:

\[
Z_i = f_i(X_1,\ldots X_i)
\]

for some function $f_i$, and

\[
\E[X_{i+1}|Z_i,\ldots, Z_1] = Z_i.
\]

A martingale is said to have the $C$-bounded difference property if

\[
|X_{i+1} - X_{i}| \leq C.
\]

The variance of a martingale is the quantity:

\[
\sigma^2 = \sum_{i \in [n-1]} \sup_{(Z_1,\ldots, Z_i)} \E[(X_{i+1} - X_i)^2| Z_1,\ldots, Z_i]
\]

We get good large deviation inequalities for martingales with bounded differences and variances (see, for example,~\cite{CL06}, Theorem 6.3 and Theorem 6.5). For a martingale $X_0, X_1,\ldots, X_n$ with respect to $Z_0, Z_1,\ldots, Z_n$, with the $C$-bounded difference property and variance $\sigma^2$, we have

\begin{equation}
\label{ineq:BdVar}
\Pr(|X_n - X_0| \geq \lambda ) \leq e^{-\frac{t^2}{2(\sigma^2 + C\lambda/3)}}
\end{equation}

\section{Discrepancy of $\mathcal{H}_1$}
\label{sec:H1}

\subsection{Proof overview}

Theorem~\ref{thm:H1} is proved by adapting Spencer's approach to color the projective plane~\cite{SPE88}, which is along the lines of the original ``six standard deviations" approach~\cite{SPE85}. The intuition is as follows:  First, observe that all the edges have size $\Theta(t)$. The analysis proceeds in two stages. Throughout (in the first and second stages), we repeatedly apply a `partial coloring' of the vertices. This just means that assigning a fractional coloring to the vertices such that a constant  fraction, say half of the vertices get $\pm 1$ valued colors. One can do this while ensuring that every edge has discrepancy at most $O(\sqrt{t})$. This step is common to many results of similar nature. A naive discrepancy bound of this process is $O(\sqrt{t} \log n)$, since in each step, a constant fraction of vertices are being colored. The main point here is that, the `randomness' of the hypergraph helps by ensuring that a half of the vertices occupies around half of most edges, and so this step reduces almost all edges evenly. This is precisely what happens in the projective plane also. After several iterations of this partial coloring, as the average edge size gets smaller, say below $t^{1 - \epsilon}$, we can tolerate a much larger discrepancy at each step, say $O(t^{\frac{1}{2} - \frac{\epsilon}{4}})$ (note that this is much more than the standard deviation). The point is that we choose parameters such that the $t^{\frac{1}{2} - \frac{\epsilon}{4}}\log n = O(\sqrt{t})$. This means that during the second phase of coloring, which lasts for at most $O(\log n)$ steps (since in each step, we color a constant fraction of remaining vertices), we never incur a higher discrepancy than $O(\sqrt{t})$ on these edges. But, there will also be `large' edges at every stage of the coloring but, the partial coloring lemma of Lovett and Meka can help ensure that these sets have discrepancy $0$ in each step as long as they are large, and so they cause no problem. This is made formal below.

First, we will observe that all the set sizes are $\Theta(t)$ with high probability.

\begin{claim}
\label{claim:uniform}
We have that with probability $1 - o(1)$, every $e \in E(\mathcal{H}_1)$, has size at most in the interval $[(t/2), (3t/2)]$.
\end{claim}

\begin{proof}
We note that the size of each set is just a $\operatorname{Bin}(n, (t/n))$ binomial random variable. Therefore, the probability that this random variable exceeds its mean by $t/2$ is at most $e^{-\frac{t}{12}}$. For $t = \Omega(\log^{10} n)$, this value is $o\left(\frac{1}{n}\right)$, which lets us union bound over all the edges.
\end{proof}

\subsection{Two stages of partial coloring}

To carry out this argument, we need the bound on $\max_{v \perp \overline{\mathbf{1}}, \|v\| = 1}\|Mv\|$ given by Theorem~\ref{thm:norm}, where $M$ is the incidence matrix of the hypergraph. However, this just consists of following the argument of Kahn and Szemer\'{e}di~\cite{FKS}. Since their setting is slightly different, one must ensure that the proof works for our setting as well, and we postpone it to Section~\ref{sec:norm}. The first ingredient is the aforementioned `partial coloring lemma'. However, this partial coloring ``lemma" is a constructive one, due to Lovett and Meka:

\begin{theorem}[\cite{LM15}, Theorem $4$]
\label{LovettMeka}
Given a family of sets $M_1,\ldots,M_m \subseteq [n]$, real numbers $c_1,\ldots,c_m$ such that $\sum_{i \in [m]}\exp\left(-c_i^2/16\right) \leq n/16$, and a real number $\delta \in [0,1]$, there is a vector $x \in [-1,1]^n$ such that:
\begin{itemize}
\item[1.] For all $i \in [m]$, $\langle x , \mathbbm{1}_{M_i}\rangle \leq c_i \sqrt{|M_i|}$.
\item[2.] $|x_i| \geq 1 - \delta$ for at least $n/2$ values of $i$.
\end{itemize}
Moreover, this vector $x$ can be found in randomized $\operatorname{poly}(n,m,(1/\delta))$ time.
\end{theorem}

\iffalse

We will also use a result by Kleitman~\cite{KLE66}

\begin{theorem}
\label{thm:kleitman}
Any subset of $2^{[n]}$ of size at least $2^{\alpha n}$ has diameter at least $2H(\alpha)n$.
\end{theorem}

We will assign a fractional coloring to the vertices using the above lemma so that every $$

\fi

Next, we need a corollary of Theorem~\ref{thm:norm}, which informally says that a $\alpha$-fraction of vertices occupy around an $\alpha$-fraction of most edges. This is a standard `pseudorandom' property of such matrices.

\begin{lemma} 
\label{lem:pseudorandom}
For any $S \subseteq V$ with $|S| = \alpha n$ and a positive real number $K$, there is a subset $E' \subset E$ of size at most $\frac{Cn}{K^2 \alpha t}$ such that for every $e \not \in E'$, we have $|e \cap S| \leq \alpha |e| + K \alpha t$, where $C = \frac{1}{\sqrt{t}}\max_{v \perp \overline{\mathbf{1}}, \|v\| = 1}\|Mv\|$
\end{lemma}

\textit{Remark:} Actually, for random $t$-regular hypergraphs, something stronger is true, i.e., the vector $(S \cap e)_{e \in E}$ very roughly looks like a gaussian vector with mean $\alpha t$, i.e., the `bad' set $E'$ is typically much smaller than $\frac{n}{K^2 \alpha t}$ as $K$ grows. However, this additional structure does not seem to offer any advantage here.

\begin{proof}
Consider a vector $v$ such that for $i \in S$, $v(i) = 1-\alpha$ and for $v \not \in S$, $v(i) = -\alpha$. We have that 

\begin{align*}
O(t\|v\|) & \geq \|Mv\|^2 \\
& = \sum_{e \in E}((1 - \alpha)|S \cap e| - \alpha|e \setminus S|)^2 \\
& = \sum_{e \in E}(|S \cap e| - \alpha|e|)^2
\end{align*}

Or 

\[
\frac{1}{n}\sum_{e \in E}(|S \cap e| - \alpha|e|)^2 \leq C\cdot t \alpha(1 - \alpha)
\]

So, the number of edges $e$ such that $|S \cap e| \geq \alpha |e| + K \alpha t$ is at most the number of $e$ such that $|S \cap e| - \alpha|e| \geq K\alpha t$, which is at most (by Markov) $n \frac{C \cdot t \alpha(1 - \alpha)}{K^2\alpha^2t^2} \leq \frac{Cn}{K^2\alpha t}$.

\end{proof}

\subsubsection{First stage of partial coloring}

\begin{lemma}
\label{lem:FirstStage}
For a subset $ \subset [n]$ such that $|S| = \alpha n$ where $t^{0.4} \leq \alpha \leq 1$, there is a partial coloring $\chi:S \rightarrow [-1,1]$ such that:
\begin{itemize}
\item[1.] $|\chi(i)| \geq 1 - \frac{1}{n}$ for at least $|S|/2$ values of $i$.
\item[2.] $|\chi(S \cap e)| \leq 20\sqrt{\alpha t \ln(2/\alpha)}$ for all $e \in E$.
\end{itemize}
\end{lemma}

\begin{proof}
Recall that by Claim~\ref{claim:uniform}, we have that $|e| = \Theta(t)$ for all $e \in E$. Call an edge $e$ small if $|S \cap e| \leq \alpha |e| + 2C \sqrt{t} \leq 2\alpha |e|$ and large otherwise. We have, by Lemma~\ref{lem:pseudorandom}, that the number of large edges is at most $\frac{\alpha n}{4}$. Consider real numbers $c_1,\ldots, c_n$ such that $c_i = 0$ if $e_i$ is large, and some large constant $10 \sqrt{\ln (2/ \alpha)}$ otherwise. One can verify that $\sum_{i \in [n]}\exp(-c_i^2/16) \leq \alpha n/16$, and so by Theorem~\ref{LovettMeka}, setting $\delta = \frac{1}{n}$, there is a coloring $\chi:S \rightarrow [-1,1]$ that satisfies the following conditions:

\begin{itemize}
\item[1.] $|\chi(i)| \geq 1 - \frac{1}{n}$ for at least $|S|/2$ values of $i$.
\item[2.] $|\chi(S \cap e)| \leq c_i \cdot \sqrt{|S \cap e|\ln(2/\alpha)}$ for all $e \in E$ for some fixed constant $K$.
\end{itemize}

Noting that either $c_i = 10$ and $|S \cap e| \leq 2\alpha |e|$ which is at most (by Claim~\ref{claim:uniform}) $4\alpha t$, or $c_i = 0$, we have that $\chi$ is the required coloring.
\end{proof}

\subsubsection{Second stage of partial coloring}

We show that when $\alpha$ gets smaller, we have even little to worry about, since we are now guaranteed discrepancy of $t^{0.3}$ in each step. Note that the procedure remains the same as above, just that the bound that is guaranteed is different.

\begin{lemma}
\label{lem:SecondStage}
For a subset $ \subset [n]$ such that $|S| = \alpha n$ where $\alpha \leq t^{0.4}$, there is a partial coloring $\chi:S \rightarrow [-1,1]$ such that:
\begin{itemize}
\item[1.] $|\chi(i)| \geq 1 - \frac{1}{n}$ for at least $|S|/2$ values of $i$.
\item[2.] $|\chi(S \cap e)| \leq K \cdot t^{0.4}$ for all $e \in E$ for some fixed constant $K$.
\end{itemize}
\end{lemma} 

\begin{proof}
Similar to the previous proof, call an edge $e$ small if $|S \cap e| \leq \alpha |e| + 2C \sqrt{t} + \alpha t \leq O(t^{0.6})$ and large otherwise. Again we have that the number of large edges is at most $\frac{\alpha n}{4}$, and setting $c_i = 0$ if $e_i$ is large, and some large constant $10 \sqrt{\ln (2/ \alpha)}$ otherwise. We have that $\sum_{i \in [n]}\exp(-c_i^2/16) \leq \alpha n/16$, and so, again, by Theorem~\ref{LovettMeka}, setting $\delta = \frac{1}{n}$, there is a coloring $\chi:S \rightarrow [-1,1]$ that satisfies the conditions:

\begin{itemize}
\item[1.] $|\chi(i)| \geq 1 - \frac{1}{n}$ for at least $|S|/2$ values of $i$.
\item[2.] $|\chi(S \cap e)| \leq 20 \cdot t^{0.3}\sqrt{\ln (2/ \alpha)}$ for all $e \in E$
\end{itemize}

Using the fact that $\sqrt{\ln (1/\alpha)} \leq \sqrt{\log n} = O(t^{0.1})$, we get the desired coloring.
\end{proof}

\subsection{Proof of Theorem~\ref{thm:H1} modulo Theorem~\ref{thm:norm}}

\begin{proof}[Proof of Theorem~\ref{thm:H1}:]
We say that a vertex $i$ `gets colored' if $|\chi(i)| \geq 1 - \frac{1}{n}$. This way, we end up with a partial coloring at the end, and round the colorings to the closest integer values to increase the discrepancy of each edge by at most $1$. So, we color the vertices of the hypergraph in two stages as follows:

\emph{First stage:} Start coloring by setting $V =: S_0$, and using Lemma~\ref{lem:FirstStage} repeatedly. Set the leftover vertices from the colorings in the $i^{th}$ step as $S_{i+1}$. Because each stage in the coloring leaves at most half the remaining vertices as uncolored, we have that $|S_i| \leq 2^{-i}n$. Therefore, the discrepancy incurred as long as the number of vertices is large is at most $\frac{n}{t^{0.4}}$ (let this be step $u$) is at most:
\begin{align*}
\sum_{i = 0}^{u}20\sqrt{t2^{-i}\ln \frac{2}{2^{-i}}} & \leq \sum_{i = 0}^{\infty}20\sqrt{t2{-^i}\ln \frac{2}{2^{-i}}} \\
& = 20\sqrt{t}\sum_{i = 0}^{\infty}\sqrt{2^{-i}\ln \frac{2}{2^{-i}}} \\
& = O(\sqrt{t})
\end{align*}

\emph{Second stage:} We reach the second stage of the coloring once the number of uncolored vertices is less than $\frac{n}{t^{0.4}}$. Lemma~\ref{lem:SecondStage} ensures that here, we can color so that all the remaining edges get an additional discrepancy at most $O\left(t^{0.4}\right)$. We can do this for at most $O(\log n)$ steps since a constant fraction of vertices are colored each step. We assumed that $t > \log^{10}n$, and so that the resulting discrepancy is at most $Kt^{0.4} \log n = O(\sqrt{t})$.

Finally, we note that every partial coloring can be obtained in time $\operatorname{poly}(n)$, and thus have an algorithmic procedure to find the coloring of guaranteed discrepancy.
\end{proof}

\section{Discrepancy of $\mathcal{H}_2$}
\label{section:H2}

\paragraph{Notation:} Throughout this section, we will use $G$ to denote the hypergeometric random variable $G(n,(n/2),(n/2))$. We observe that $\Pr(G = k) = \binom{n/2}{k}^2$

Let $M$ be the random $m \times n$ matrix where each entry is $0$ or $1$ with probability $\frac{1}{2}$ independently. Recall that $m = \frac{n}{C_3 \log n}$ where $C_3$ is a large enough constant to be specified later. We say a coloring $u \in \{\pm 1\}^n$ is \emph{good} for a row $e$ of $M$ if $\langle u, e \rangle = 0$. For simplicity, we further condition that each row of $M$ has an even number of $1$'s and estimate on the probability that the discrepancy of this is $0$. The general case can be handled by first sampling rows of even weights, and then resampling the first element of every row. This resampling increases the discrepancy by at most $1$, and the procedure generates a randomly chosen $0,1$ matrix. More concretely, we have:

\begin{proposition}
\label{proposition:even}
Let $\mathcal{E}$ denote the event that every row in $M$ has an even number of $1$'s. Suppose we have $\Pr(\operatorname{disc}(M) \leq t | \mathcal{E}) = 1 - o(1)$, then $\Pr(\operatorname{disc}(M) \leq t + 1) = 1 - o(1)$
\end{proposition}

Henceforth, we condition that every edge has even size.

We use $\chi_1,\chi_2, \ldots$ to denote balanced colorings (i.e., colorings containing the same number of $+1$'s and $-1$'s). Let $X$ be the number of good balanced colorings for the randomly chosen hypergraph.

\begin{claim}
\label{claim:exp}
We have:
\[
\E[X] = \binom{n}{n/2}\cdot \left(2\sqrt{\frac{2}{\pi n}}\right)^m\left(1 + O\left(\frac{1}{\log n}\right)\right)
\]
\end{claim}

\begin{proof}

Define the events:

\[
G_i = \{\chi_i \text{ is good for a randomly chosen row}\}.
\]

We have:

\begin{align*}
\Pr(G_1) & = \frac{1}{2^{n-1}}\sum_{i = 0}^{n/2}\binom{n/2}{i}^2 \\
& = \frac{1}{2^{n-1}}\binom{n}{n/2} \\
& = \frac{1}{2^{n-1}} \cdot 2^n \cdot \sqrt{\frac{2}{\pi n}}\left( 1 + O\left( \frac{1}{n}\right)\right) \\
& = 2\sqrt{\frac{2}{\pi n}}\left( 1 + O\left( \frac{1}{n}\right)\right)
\end{align*}

Since the edges are picked independently, by linearity of expectation, and using that $m \lesssim \frac{n}{\log n}$, we have $\E[X] = \binom{n}{n/2}\cdot \left(2\sqrt{\frac{2}{\pi n}}\right)^m\left(1 + O\left(\frac{1}{\log n}\right)\right)$. 

\end{proof}

\subsection{Second moment}

The aim of this subsection is to prove the following:

\begin{lemma}
\label{lem:secmom}
We have:
\[
\E[X^2] = (1 + o(1))\E[X]^2
\]
\end{lemma}

As a start, we will look at the probability that a pair of colorings $\chi_1$ and $\chi_2$ are both good for $M$.
Again, define the events:

\[
G_i = \{\chi_i \text{ is good for a randomly chosen row}\}.
\]

\begin{claim}
For  pair of balanced colorings $\chi_1, \chi_2$ that have Hamming distance $\alpha n$, where $\alpha$ is some fixed constant which is at least $0.25$, we have:

\[
\Pr(G_1 G_2) = \frac{4}{\sqrt{\alpha(1 - \alpha)}\pi n} \left( 1 + O\left( \frac{1}{n}\right)\right)
\]
\end{claim}

\begin{proof}

We simply evaluate:

\begin{align*}
\Pr(G_1  G_2) & = \frac{1}{2^{n-1}} \cdot \left(\sum_{i = 1}^{n/4}\binom{\alpha n/2}{i}^2\right)\cdot \left( \sum_{i = 1}^{n/4}\binom{(1 - \alpha)n/2}{i}^2\right) \\
& = \frac{1}{2^{n-1}}\cdot \binom{\alpha n}{\alpha n/2}\cdot \binom{(1 - \alpha)n}{(1 - \alpha)n/2} \\
& = \frac{1}{2^{n-1}}\cdot 2^{n}\cdot \frac{2}{\pi} \cdot \sqrt{\frac{1}{(\alpha)(1 - \alpha)}}\left( 1 + O\left( \frac{1}{n}\right)\right) \\
& = \left( \frac{4}{\sqrt{\alpha(1 - \alpha)}\pi n} \right)\left( 1 + O\left( \frac{1}{n}\right)\right)
\end{align*}

\end{proof}

One observation is that $\alpha$ is close to $\frac{1}{2}$, then $\Pr(G_1G_2)$ is asymptotically $\Pr(G_1)^2$. One might hope that most of the contribution to the second moment come from such pairs (which are also the most common type of pairs). For convenience, let us define another parameter $t$ and henceforth, assume that $\alpha$ and $t$ are related as $\alpha = \frac{1}{2} - \frac{t}{n}$. The above claim states that for $|t| \leq n/4$:

\begin{align*}
\Pr(G_1G_2) & \sim \left( \frac{8}{\pi n}\right) \cdot \left( \frac{1}{\sqrt{1 - \frac{4t^2}{n^2}}}\right)\left( 1 + O\left(\frac{1}{n}\right)\right)\\
& = \left( \frac{8}{\pi n}\right) \cdot \left(1 - \frac{4t^2}{n^2} \right)^{-\frac{1}{2}}\left( 1 + O\left(\frac{1}{n}\right)\right) \\
\end{align*}

Therefore, after making sure that the error term does not really change the asymptotics, we have that the probability that $\chi_1$ and $\chi_2$ are both good for $m$ independently picked hyperedges is 

\begin{equation}
\label{eqn:largealpha}
\left( \frac{8}{\pi n}\right)^m \cdot \left(1 - \frac{4t^2}{n^2} \right)^{-\frac{m}{2}} \leq \left( \frac{8}{\pi n}\right)^m \cdot e^{\frac{4t^2}{C_3 n \log n}}
\end{equation}

where the inequality is due to the fact that $1 - x \geq e^{-2x}$ holds when $x < 0.25$. The other case, $|t| > n/4$, is handled by the trivial bound:

\begin{align*}
\Pr(G_1G_2) & \leq \Pr(G_1) \\
& = \sqrt{\frac{8}{\pi n}}
\end{align*}

And thus, the probability that two such colorings are good for $M$ is at most:

\begin{equation}
\label{eqn:smallalpha}
\left(\frac{8}{\pi n}\right)^{\frac{m}{2}} = \left(\frac{8}{\pi n}\right)^{m} \left(\frac{8}{\pi n}\right)^{-\frac{m}{2}}  \leq \left(\frac{8}{\pi n} \right)^m e^{\frac{n}{16}}
\end{equation}

for large enough $C_3$. On the other hand, the number of ordered pairs of colorings at hamming distance $\alpha n$ is precisely

\begin{align*}
\binom{n}{n/2} \cdot \binom{n/2}{\alpha/2}^2 
\end{align*}

\begin{proof}[Proof of Lemma~\ref{lem:secmom}:]

Proceeding in a straightforward matter, using~(\ref{eqn:largealpha}) and~(\ref{eqn:smallalpha}), we have:

\begin{align*}
\E[X^2] & \leq \sum_{\alpha n / 2 = 0}^{n/8} \binom{n}{n/2} \cdot \binom{n/2}{\alpha n/2}^2 \left( \frac{4}{\sqrt{\alpha(1 - \alpha)}\pi n} \right)^m\left(1 +  O\left( \frac{1}{\log n}\right)\right) \\
& ~~~ + \sum_{\alpha n / 2 = n/8}^{n/2} \binom{n}{n/2} \cdot \binom{n/2}{\alpha n/2}^2  \cdot \left( \frac{8}{\pi n}\right)^m \cdot e^{\frac{n}{16}} \\
& \leq \binom{n}{n/2}\left( \frac{8}{\pi n}\right)^m \Bigg( \sum_{t/2 = -n/8}^{n/8} \binom{n/2}{n/4 - t/2}^2 \left(e^{\frac{4t^2}{C_3n\log n}}\right)  \\
& ~~~ + 2 \cdot e^{\frac{n}{16}}\sum_{t/2 = n/8 + 1}^{n/4} \binom{n/2}{n/4 - t/2}^2 \Bigg)\left(1 +  O\left( \frac{1}{\log n}\right)\right)
\end{align*}

For a rough plan, let us first ignore the $\left(1 +  O\left( \frac{1}{\log n}\right)\right)$ part in the above inequality. If we hope to prove that $\E[X^2] \sim E[X]^2$, we \emph{must} prove that 

\[
\Bigg( \sum_{t/2 = -n/8}^{n/8} \binom{n/2}{n/4 - t/2}^2 \left(e^{\frac{4t^2}{C_3n\log n}}\right) + 2 \cdot e^{\frac{n}{16}}\sum_{t/2 = n/8 + 1}^{n/4} \binom{n/2}{n/4 - t/2}^2 \Bigg) \sim \binom{n}{n/2}.
\]

Hence, we shall only care about this summation. First, we observe that the second term has very little contribution to the second moment and deal with it immediately. By the Chernoff bound for hypergeometric r.v's~(\ref{ineq:Hypergeom}):

\begin{align*}
2\sum_{t/2 = n/8 + 1}^{n/4}\binom{n/2}{n/4 - t/2}^2 = 2\cdot \binom{n}{n/2}\cdot \Pr\left(G \geq (5n/8) \right) \leq \binom{n}{n/2}\cdot 2e^{-\frac{n}{8}}.
\end{align*}

And thus the total contribution of the second term is 

\[
2 \cdot e^{n/16} \cdot \left(\frac{8}{\pi n} \right)^m \cdot \binom{n}{n/2}\cdot \sum_{t/2 = n/8 + 1}^{n/4} \binom{n/2}{n/4 - t/2}^2 = \E[X]^2\cdot o(1).
\]

For the first term, let us define the term 

\[T' = \sum_{t/2 = -n/8}^{n/8} \binom{n/2}{n/4 - t/2}^2 \left(e^{\frac{4t^2}{C_3n\log n}}\right)
\]
that we are left with the task of estimating. Actually, we are interested in $T \defeq T'/\binom{n}{n/2}$. If we prove that this term is bounded by $1 + o(1)$, then we are done.

Towards this, fix $t_0 = \sqrt{n}(\log n)^{\frac{1}{4}}$ and partition: 
\[
S := \{-n/8,-n/8 + 1, \ldots, n/8 - 1, n/8\} = A \sqcup B_1 \sqcup \cdots \sqcup B_{\ell}
\]

as follows:

\begin{itemize}
\item[1.] $A = \{x \in S \suchthat |x| \leq t_0\}$
\item[2.] $B_i = \{x \in S \suchthat t_0 + (i-1) \cdot \sqrt{n} \leq |x| < t_0 +  i \cdot \sqrt{n}\}$.
\end{itemize}

Also, abbreviate $t_0 + i\cdot \sqrt{n} =: t_i$ for $i = 0,\ldots, \ell$. Now we rewrite our quantity of interest:

\begin{align*}
T & = \frac{\sum_{t/2 \in A} \binom{n/2}{n/4 - t/2}^2 \left(e^{\frac{4t^2}{C_3n\log n}}\right)}{\binom{n}{n/2}} + \frac{\sum_{t/2 \in B} \binom{n/2}{n/4 - t/2}^2 \left(e^{\frac{4t^2}{C_3n\log n}}\right)}{\binom{n}{n/2}}\\
\end{align*}

Claim~\ref{claim:sumA} and Claim~\ref{claim:sumB} finish the proof.

\begin{claim}
\label{claim:sumA}
We have:

\[
\frac{\sum_{t/2 \in A} \binom{n/2}{n/4 - t/2}^2 \left(e^{\frac{4t^2}{C_3n\log n}}\right)}{\binom{n}{n/2}} = 1 - o(1).
\]
\end{claim}

\begin{proof}

Using large deviation bounds for hypergeometric random variables (\ref{ineq:Hypergeom}), we have:

\begin{align*}
\sum_{|t/2| \leq t_0}\frac{\binom{n/2}{n/4 - t/2}^2 }{\binom{n}{n/2}}\left(e^{\frac{4t^2}{C_3n\log n}}\right) & \leq e^{\frac{4t_0^2}{C_3n\log n}} \cdot \sum_{|t/2| \leq t_0}\frac{\binom{n/2}{n/4 - t/2}^2 }{\binom{n}{n/2}} \\
& = e^{\frac{4t_0^2}{C_3n\log n}} \cdot \Pr\left(G \in n/4 \pm t_0\right) \\
&\leq  e^{\frac{t_0^2}{C_3n \log n}}\left(1 - e^{-\frac{2t_0^2}{n}}\right) \\
& = 1 - O\left(\frac{1}{\sqrt{\log n}}\right).
\end{align*}

\end{proof}

Next, we bound the terms corresponding to each of the $B_i$'s separately, in a manner similar as above. First, one does need to verify that the probability that hypergeometric random variable lies in the interval $t \pm \sqrt{n}$ away from its mean is of the order $ e^{O\left(\frac{t^2}{n}\right)}$.

For convenience, let us define some additional notation: 

\[
p_i \defeq \Pr\left(\left|G - (n/4)\right| \in B_i\right) = \frac{\sum_{t/2 \in B_i} \binom{n/2}{n/4 - t/2}^2}{\binom{n}{{n/2}}}
\]

is the probability that $t \in B_i$, and

\begin{align*}
q_i & \defeq \frac{\sum_{t/2 \in B_i} \binom{n/2}{n/4 - t/2}^2\left(e^{\frac{4t^2}{C_3n\log n}}\right)}{\binom{n}{{n/2}}} \\
\end{align*}

is the contribution to $T$ when the summation is over $B_i$. We are eventually interested in estimating the sum $\sum_{j = 1}^{\ell}q_i$. We will first prove the following:

\begin{claim}
\label{claim:pq}
There is a constant $C''$ such that 

\[
q_i = C'' \cdot p_i^{1 - O\left(\frac{1}{\log n} \right)}
\]
\end{claim}

\begin{proof}

Consider a generic $p_i$ and $q_i$. We have $\Pr(G = \alpha n/2) = \binom{n/2}{\alpha n/2}^2$. Suppose that $\alpha = \frac{1}{2} - \frac{t_i}{n}$. Here, we have the additional constraint that $|t_i| \geq t_0 \gg \sqrt{n}$. We would like to get a lower bound on the quantity:

\[
p_i = \Pr(X \in [\alpha n/2, \alpha n/2 + \sqrt{n}])
\]

Proceeding in the naive way, we have:

\begin{align*}
p_i & =  \Pr(G \in [\alpha n/2, \alpha n/2 + \sqrt{n}]) \\
&\geq \sqrt{n} \cdot \Pr(G = \alpha n/2) \\
&= \sqrt{n} \cdot \binom{n/2}{\alpha n/2}^2 \Bigg / \binom{n}{n/2} \\
& \sim \sqrt{n} \cdot \frac{4}{\pi n \alpha(1 - \alpha)} 2^{H(\alpha)n} \Bigg / \sqrt{\frac{2}{\pi n}} 2^n \\
& = \frac{2}{\alpha(1 - \alpha)} \sqrt{\frac{2}{\pi}} \cdot 2^{(H(\alpha) - 1)n}
\end{align*}

Plus, we have that $(H(\alpha) - 1) = H\left(\frac{1}{2} - \frac{t_i}{n}\right) - 1 \geq -5\left(\frac{t_i}{n}\right)^2$, and this $p_i > C' \cdot 2^{- 5\frac{t_i^2}{n}}$,

On the other hand, we have the following estimate for $q_i$:

\begin{align*}
q_i &= \frac{\sum_{t/2 \in B_i} \binom{n/2}{n/4 - t/2}^2\left(e^{\frac{4t^2}{C_3n\log n}}\right)}{\binom{n}{{n/2}}} \\
& \leq e^{\frac{4(t_0 + i\cdot \sqrt{n})^2}{Cn\log n}}\frac{\sum_{t/2 \in B_i} \binom{n/2}{n/4 - t/2}^2}{\binom{n}{{n/2}}} \\
& = e^{\frac{4t_i^2}{Cn\log n}}p_i.
\end{align*}

Thus we have the relation between $p_i$ and $q_i$ given by

\begin{equation}
\label{eqn:pq}
q_i \leq C'' \cdot p_i^{1 - O\left(\frac{1}{\log n}\right)}
\end{equation}

for some constant $C''$, as desired.
\end{proof}

Now we can bound the second term in our sum for $T$.

\begin{claim}
\label{claim:sumB}
We have

\[
\frac{\sum_{t/2 \in B} \binom{n/2}{n/4 - t/2}^2 \left(e^{\frac{4t^2}{Cn\log n}}\right)}{\binom{n}{n/2}} = o(1).
\]
\end{claim}

\begin{proof}

First, we observe that $\ell = O(\sqrt{n})$, and so 

\begin{equation}
\label{eqn:ell}
\ell^{O\left(\frac{1}{\log n}\right)} = O(1). 
\end{equation}

With this in mind, we compute:

\begin{align*}
\frac{\sum_{t/2 \in B} \binom{n/2}{n/4 - t/2}^2 \left(e^{\frac{4t^2}{Cn\log n}}\right)}{\binom{n}{n/2}} & = \sum_{i = 1}^{\ell} q_i \\
& \leq C''\sum_{i = 1}^{\ell}p_i^{1 - O\left(\frac{1}{\log n} \right)} \\
& \leq C'' \cdot \ell \cdot \left( \frac{\sum_{i = 1}^{\ell}p_i}{\ell}\right)^{1 - O\left(\frac{1}{ \log n} \right)} \\
& \leq C'' \cdot \ell^{O\left(\frac{1}{\log n} \right)} \cdot e^{-\frac{2t_0^2}{n}\left(1 - O\left(\frac{1}{\log n} \right)\right)} \\
& = e^{-\Omega(\sqrt{\log n})}
\end{align*}

In the above chain of inequalities, we used Claim~\ref{claim:pq}, Jenson's inequality, the fact that $\sum p_i \leq e^{-\frac{t_0^2}{n}}$, and~(\ref{eqn:ell}) respectively.

\end{proof}

\end{proof}

\begin{proof}[Proof of Theorem~\ref{thm:H2}]
First, we condition on the event that every edge in $\mathcal{H}_2$ has even size. We want to lower bound the probability that $X$, the number of good balanced colorings, is nonzero. Lemma~\ref{lem:secmom} gives us that $\E[X^2] = (1 + o(1))\E[X]^2$, and Corollary~\ref{corr:main} gives us that 
\[
\Pr(X = 0) = o(1).
\]
Using Proposition~\ref{proposition:even}, we have that the discrepancy of $\mathcal{H}_2$ is, with high probability, at most $1$.
\end{proof}

\section{`Norm' of $M$: Proof of Theorem~\ref{thm:norm}}
\label{sec:norm}

In this section, we shall prove Theorem~\ref{thm:norm}. As mentioned before, this just means verifying the proof of Kahn and Szemer\'{e}di for our random model (also see~\cite{BFSU98}) We assume that the regularity is $t \ll n^{1/2}$. 

We shall prove that for every $x$, and $y$ such that $\|x\| = \|y\| = 1$ and $x \perp \overline{1}$, we have that $|y^tMx| \leq O(\sqrt{t})$. First, we `discretize' our problem by restricting $x$ to belong to the $\epsilon$-net:

\[
T \defeq \left\{x \in \left( \frac{\epsilon}{\sqrt{n}}\mathbb{Z}\right)^n \suchthat~\|x\| \leq 1 \text{ and } x \perp \overline{1}\right\}
\]

and $y$ belonging to 

\[T' \defeq \left\{y \in \left( \frac{\epsilon}{\sqrt{n}}\mathbb{Z}\right)^n \suchthat \|y\| \leq 1 \right\}
\]

for a small enough constant $\epsilon$.

\begin{claim}[\cite{FKS}, Proposition 2.1)]
If for every $x \in T$, and $y \in T'$, we have that $\|y^tMx\| \leq B$ for some constant $C$, then we have that for every $z \in \R^n$ such that $\|z\| = 1$, we have that $\|Mz\| \leq (1 - 3\epsilon)^{-1}B$
\end{claim}

\begin{proof}
Let $z = \operatorname{argmax}_{\|z\| = 1}\|Mz\|$. We shall use the fact that there are $x \in T$, and $y \in T'$ such that $\|x - z\| \leq \epsilon$, and $\left\|y - \frac{Mz}{\|Mz\|} \right\| \leq \epsilon$. With this in mind, we have:

\begin{align*}
\|Mz\| & = \left\langle \frac{Mz}{\|Mz\|}, Mz \right\rangle \\
& = \langle y + w_1, M(x+w_2) \rangle \\
& = y^tMx + \langle w_1, Mx \rangle + \langle y,Mw_2 \rangle + \langle w_1 ,Mw_2\rangle
\end{align*}

Where $|w_1|,|w_2| \leq \epsilon$. We note that each of the terms $\langle w_1, Mx \rangle$ and $\langle y,Mw_2 \rangle$, and $\langle w_1 ,Mw_2\rangle$ are upper bounded by $\epsilon \|Mz\|$, and $\langle w_1 ,Mw_2 \rangle \leq \epsilon^2 \|Mz\|$. Combining this, and using the fact that $\epsilon^2 \leq \epsilon$, we have
\[
\|Mz\| \leq (1 - 3\epsilon)^{-1}y^tMx \leq (1 - 3\epsilon)^{-1}B
\]
\end{proof}

So now, will need to only union bound over $T \cup T'$. It is not hard to see that each of these has size at most $|T|,|T'| \leq \left(\frac{C_v}{\epsilon}\right)^n$ for some absolute constant $C_v$.

Indeed, we have:

\begin{align*}
|T| & \leq \left(\frac{\sqrt{n}}{\epsilon}\right)^n\operatorname{Vol}\left\{x \in \R^n \suchthat \|x\| \leq 1+ \epsilon\right\} \\
& \leq \left(\frac{\sqrt{n}}{\epsilon}\right)^n \cdot \frac{1}{\sqrt{\pi n}} \left( \frac{2 \pi e}{n}\right)^{n/2}(1+\epsilon)^n \\
& \leq \left(\frac{C_v}{\epsilon}\right)^n
\end{align*}

For some constant $C_v$.

We split the pairs $[n] \times [n] = L \cup \overline{L}$ where $L \defeq \{(u,v) \suchthat |x_uy_v| \geq \sqrt{t}/n\}$, which we will call `large entries' and write our quantity of interest:

\begin{align*}
\sum_{(u,v) \in [n] \times [n]} x_uM_{u,v}y_v = \sum_{(u,v) \in L}x_uM_{u,v}y_v + \sum_{(u,v) \in \overline{L}}x_uM_{u,v}y_v
\end{align*}

\textbf{For the large entries}: For a set of vertices $A \subset [n]$ and a set of edges $B \subset [n]$, let us denote $I(A,B)$ to be the number of vertex-edge incidences in $A$ and $B$. Let us use $\mu(A,B) \defeq \E[|I(A,B)|]$.

\begin{lemma}
\label{lem:dense}
We have, for every set $A$ of points and every set $B$, we have that with high probability, $I \defeq |I(A,B)|$ and $\mu \defeq \mu(A,B)$ satisfy at least one of the following:
\begin{enumerate}
\item $I \leq 2 \mu$
\item $I \log \left(I/\mu \right) \leq C |B| \log \left( n / |B| \right)$
\end{enumerate}
\end{lemma}

This lemma is sufficient to show that the large pairs do not contribute too much, as shown by the following lemma, which is the main part of the proof of Kahn and Szemer\'{e}di.

\begin{lemma}[\cite{FKS}, Lemma 2.6] 
If the conditions given in Lemma~\ref{lem:dense} are satisfied, then $\sum_{(u,v) \in L} \left|x_u M_{u,v}y_v \right| = O(\sqrt{t})$ for all $x,y \in T$.
\end{lemma}

Notice that since we are bounding $\sum_{(u,v) \in L} \left|x_u M_{u,v}y_v \right| = O(\sqrt{t})$, which is much stronger than what we really need, it is okay to consider both $x$ and $y$ from $T$.

\begin{proof}[Proof of Lemma~\ref{lem:dense}]
Let $\mathcal{B}_i(a,b)$ denote the event that there is an $A$ and a $B$ which do not satisfy either of the conditions and $|I(A,B)| = i$. Before, we prove the lemma, let us make some observations, which (in hindsight) help us compute the probabilities much easier. Let $A$ be a set of $a$ vertices and $B$ be a collection of $b$ edges, such that $a \leq b$ and $b \leq n/2$ (this argument is symmetric in $a$ and $b$). 

The point here is that we basically want to evaluate the sum:

\begin{align*}
\Pr\left(\bigcup_{a,b,i}\mathcal{B}_i(a,b) \right) & \leq \sum_i \Pr\left(\bigcup_{a,b}\mathcal{B}_i(a,b)\right) \\
& = \sum_{i \leq \log ^2n}\Pr\left(\bigcup_{a,b}\mathcal{B}_i(a,b)\right) + \sum_{i \geq \log^2 n}\Pr\left(\bigcup_{a,b}\mathcal{B}_i(a,b)\right)
\end{align*}

The first observation is that every term in the second sum is small. Towards this, we have the straightforward claim.

\begin{claim}
\label{claim:fixed}
For a set of vertices $A$ and edges $B$ and a set of possible incidences $J \subset A \times B$, we have that $\Pr(I(A,B) = J) \leq ^{}\left( \frac{2t}{n}\right)^{|J|}$
\end{claim}

\begin{proof}
W.L.O.G, let $A = \{1,\ldots, a\}$, and for $i \in A$, let $t_i = I(\{i\},B)$. We have that:

\begin{align*}
\Pr(I(A,B) = J) & = \prod_{i \in A}\frac{\binom{n-b}{t- t_i}}{\binom{n}{t}} \\
& \sim \frac{(n - b)^{at - |J|}}{n^{at}} \cdot \frac{(t!)^a}{\prod_{i \in A}(t - t_i)!} \\
& \leq e^{-\mu(A,B)}\left( \frac{n}{2}\right)^{-|J|} \cdot (t)^{|J|} \\
& \leq \left(\frac{2t}{n}\right)^{|J|}
\end{align*}
\end{proof}

If $\mathcal{B}_i(a,b)$ holds for some $A$ and $B$, we have that $I \log (I/\mu) > Cb \log(n/b)$, or 

\[
\left( \frac{n}{b}\right)^{Cb} \cdot \left( \frac{\mu}{I} \right)^I \leq 1
\]

Therefore, we have:

\begin{align*}
\Pr\left(\bigcup_{a,b}\mathcal{B}_i(a,b)\right) & \leq \binom{n}{a}\binom{n}{b}\binom{ab}{i}\left(\frac{2d}{n}\right)^i \\
& \leq \binom{n}{b}^2\left(e\frac{abd}{ni}\right)^i \\
& \leq \left(\frac{en}{b}\right)^{2b} \left( e\frac{\mu}{i}\right)^i \\
& \ll \left(\frac{b}{n}\right)^{C'b} \\
& \ll n^{-C'}
\end{align*}

for some large enough constant $C'$. Therefore, the second sum is at most $n^2 \frac{1}{n^{C'}} = o(1)$

It remains to deal with the sum $\sum_{i \leq \log ^2n}\Pr\left(\bigcup_{a,b}\mathcal{B}_i(a,b)\right)$. For these summands, we have that if $|I(A,B)| \leq \log^2n$ and $I \log (I/\mu) > Cb \log(n/b)$, then $I \geq Cb$. Therefore, we only need to evaluate the sum:

\begin{align*}
\sum_{i = Cb}^{\log^2 n}\Pr\left(\bigcup_{a,b}\mathcal{B}_i(a,b)\right) & \leq \binom{n}{a}\binom{n}{b}\sum_{i = Cb}^{\log^2 n}\binom{ab}{i}\left( \frac{2d}{n}\right)^i \\
& \leq \binom{n}{b}^2\sum_{i = Cb}^{\log^2 n} \left( \frac{2eabd}{in}\right)^i \\
& \leq \log^2n \left( \frac{n}{b} \right)^{2b}\left(\frac{2be}{Cn}\right)^{Cb} \\
& = o(1)
\end{align*}

We have used the assumption that $b \leq n/2$ in Claim~\ref{claim:fixed}. It can be easily checked that when, $b > n/2$, then $|I(A,B)| \leq d|A| \leq 2\mu(A,B)$.

\end{proof}

\textbf{For the small entries:}

Bounding the contribution from the small entries is much easier. The analysis given here is slightly different to the one given in~\cite{FKS} and~\cite{BFSU98}. However, it does not make much of a difference, and is still, essentially, a large deviation inequality. We will first compute the expected value of the quantity of interest using the following claim:

\begin{claim}
We have that:
\[
\left| \sum_{u,v \in \overline{L}} x_u y_v\right| \leq \frac{n}{\sqrt{t}}
\]
\end{claim}

\begin{proof}
Since $\sum x_i = 0$, we have 

\[
\left( \sum x_i\right) \left(\sum y_i \right) = \sum_{(u,v) \in L} x_uy_v + \sum_{(u,v) \in \overline{L}}x_uy_v = 0
\] 

or 

\[
\left| \sum_{u,v \in \overline{L}} x_u y_v\right|  = \left| \sum_{u,v \in L} x_u y_v\right| 
\]

To bound this, we note that 
\begin{align*}
1 & = \left(\sum x_u^2\right) \left( \sum y_u^2 \right) \\
& \geq \sum_{(u,v) \in L} x_u^2y_v^2 \\
& \geq \frac{\sqrt{t}}{n}\left| \sum_{(u,v) \in L}x_uy_v \right|
\end{align*}

which gives us what we want.
\end{proof}

Given the above lemma, we can easily compute the expectation:

And so 

\[
\E\left[ \sum_{(u,v) \in L}x_uM_{u,v}y_v \right] = \frac{t}{n} \sum_{(u,v) \in \overline{L}}x_uy_v \in [- \sqrt{t}, \sqrt{t}]
\]

\begin{claim}
We have that with high probability,
\[
\sum_{(u,v) \in \overline{L}}x_uM_{u,v}y_v = O(\sqrt{t})
\]
\end{claim}

\begin{proof}
We set up a martingale and use the method of bounded variances. Let us write the quantity that we wish to estimate as 

\[
X \defeq \sum_{(u,v) \in B}x_uM_{u,v}y_v
\]

We imagine $M$ being sampled one column at a time, and in each column, $t$ entries are sampled. For column $i$, let us denote these by $e_{i,1}, \ldots, e_{i,t}$, and let us abbreviate $E_i \defeq (e_{i,1},\ldots, e_{i,t})$, and $E_{\overline{i}} \defeq (E_1,\ldots, E_{i-1}, E_{i+1}, \ldots E_n)$. Clearly, $X = X(e_{1,1},\ldots, e_{n,t})$. Denote $X_{i,j} \defeq \E[X | e_{1,1} ,\ldots, e_{i,j}]$. 

For distinct $k,k' \in [n]$, it is easy to see that we have the `Lipshitz' property:

\[
\E[X | e_{1,1},\ldots, e_{i,j-1}, e_{i,j} = k] - \E[X | e_{1,1},\ldots, e_{i,j-1}, e_{i,j} = k'] \leq |x_iy_k| + |x_iy_{k'}|
\]

Therefore, we have a bounded difference property on $|X_{i,j} - X_{i,j-1}|$ as follows:

\begin{align*}
|X_{i,j} - X_{i,j-1}| & = \Bigg|\E[X | e_{1,1},\ldots, e_{i,j-1}, e_{i,j}]  \\
&~~~~ -  \frac{1}{n - j - 1}\sum_{k' \in [n] \setminus \{e_{i,1},\ldots,e_{i,j-1}\}}\E[X | e_{1,1},\ldots, e_{i,j-1}, e_{i,j} = k'] \Bigg| \\
& \leq |x_i||y_{e_j}| + \frac{1}{n - j - 1}\sum_{k' \in [n] \setminus \{e_{i,1},\ldots,e_{i,j-1}\}}\mathbb{1}[(i,k) \in \overline{B}]|x_iy_{k'}|\\
& \leq |x_iy_{e_j}| + \frac{1}{\sqrt{n}}\cdot \frac{\sqrt{t}}{n} \\
\end{align*}

We will use that the above quantity is bounded by $\frac{2\sqrt{t}}{n}$ since we only consider $|x_iy_{e_j}|$ where $(i,e_j) \in \overline{B}$.

Now, we would like to compute the variance of the martingale

\begin{align*}
\Var(X_{i,j} - X_{i,j-1} | e_{1,1},\ldots, e_{i,j-1}) & \leq \frac{1}{n - j - 1} \sum_{k \in [n]} \left(|x_iy_{k}| + \frac{\sqrt{t}}{n\sqrt{n}}\right)^2 \\
& \leq \frac{2}{n - j - 1} \sum_{k \in [n]}\left(|x_iy_k|^2 + \frac{t}{n^3} \right) \\
& \leq \frac{2x_i^2}{n - j - 1} + O\left( \frac{t}{n^4}\right)
\end{align*}

Where the last inequality uses that $\sum_k y_k^2 \leq 1$. Therefore, the variance of the martingale is at most

\[
\frac{2}{n - t}\sum_i x_i^2 + O\left( \frac{t}{n^3}\right) \leq \frac{3}{n} =: \sigma^2
\]

since $\sum_i x_i^2 \leq 1$. Therefore, by the bounded variance martingale inequality~(\ref{ineq:BdVar}), using $|X_i - X_{i-1}|\leq \frac{2\sqrt{t}}{n} =: D$:

\begin{align*}
\Pr(X \geq (C+1)\sqrt{t}) & \leq \exp \left\{ - \frac{C^2t}{2\sigma^2 + tD/3} \right\} \\
& \leq \exp \left\{ -\frac{C^2t}{\frac{3}{n} + \frac{2t}{3n}}\right\} \\
& \leq \exp\left\{ -C^2n \right\}
\end{align*}

This lets us union bound over all $x,y \in T$, whose number can be bounded by $\left( \frac{C_v}{\epsilon}\right)^n$.

\end{proof}

\section{Conclusion and future directions}
One thing to note on the algorithmic side, it is worth noting that the second result does not give an easy efficient algorithm, and it would be interesting to find one. This was asked as an open problem in~\cite{HR18} too, in the setting of Theorem~\ref{thm:H2}, when $n = \Omega(m^2 \log m)$. 

As for better bounds in the random model, one natural question is whether or not one can easily verify the Beck-Fiala conjecture for all values of $t$. At first, it is tempting to claim that the second moment method of Theorem~\ref{thm:H2} might  straightaway yield better bounds in $\mathcal{H}_1$, i.e., to either get a bound of $O(\sqrt{t})$ for the range of parameters not handled, or to get a bound of $o(\sqrt{t})$ when $n \gg m$. However, as of now, we have not attempted this.

\section{Acknowledgements}

I am extremely thankful to Jeff Kahn for suggesting the problem, the helpful discussions, and the numerous references. I would also like to thank Huseyin Acan and Cole Franks for the very helpful discussions. I would also like to thank Cole for sharing an unfinished draft of~\cite{FS18} with me, and Amey Bhangale for helpful feedback on the writeup. I am also grateful to the anonymous SODA referees whose suggestions, among other things, helped make Theorem~\ref{thm:H1} constructive.

\bibliographystyle{alpha}
\bibliography{references.bib}

\begin{thebibliography}{BFSU98}

\bibitem[AS00]{AS00}
Noga Alon and Joel~H. Spencer.
\newblock The probabilistic method, 2000.

\bibitem[Ban10]{BAN10}
Nikhil Bansal.
\newblock Constructive algorithms for discrepancy minimization.
\newblock In {\em 51th Annual {IEEE} Symposium on Foundations of Computer
  Science, {FOCS} 2010, October 23-26, 2010, Las Vegas, Nevada, {USA}}, pages
  3--10, 2010.

\bibitem[Bec81]{BEC81}
J{\'{o}}zsef Beck.
\newblock Roth's estimate of the discrepancy of integer sequences is nearly
  sharp.
\newblock {\em Combinatorica}, 1(4):319--325, 1981.

\bibitem[BF81]{BF81}
J{\'{o}}zsef Beck and Tibor Fiala.
\newblock "integer-making" theorems.
\newblock {\em Discrete Applied Mathematics}, 3(1):1--8, 1981.

\bibitem[BFSU98]{BFSU98}
Andrei~Z. Broder, Alan~M. Frieze, Stephen Suen, and Eli Upfal.
\newblock Optimal construction of edge-disjoint paths in random graphs.
\newblock {\em {SIAM} J. Comput.}, 28(2):541--573, 1998.

\bibitem[Buk16]{BUK16}
Boris Bukh.
\newblock An improvement of the beck-fiala theorem.
\newblock {\em Combinatorics, Probability and Computing}, 25(3):380?398, 2016.

\bibitem[CL06]{CL06}
Fan Chung and Linyuan Lu.
\newblock Concentration inequalities and martingale inequalities: a survey.
\newblock {\em Internet Math.}, 3(1):79--127, 2006.

\bibitem[EL15]{EL15}
Esther Ezra and Shachar Lovett.
\newblock On the beck-fiala conjecture for random set systems.
\newblock In {\em APPROX-RANDOM}, 2015.

\bibitem[FKS89]{FKS}
J.~Friedman, J.~Kahn, and E.~Szemer{\'e}di.
\newblock On the second eigenvalue of random regular graphs.
\newblock In {\em Proceedings of the Twenty-first Annual ACM Symposium on
  Theory of Computing}, STOC '89, pages 587--598, New York, NY, USA, 1989. ACM.

\bibitem[FS18]{FS18}
C.~{Franks} and M.~{Saks}.
\newblock {On the Discrepancy of Random Matrices with Many Columns}.
\newblock {\em ArXiv e-prints}, July 2018.

\bibitem[HR18]{HR18}
Rebecca Hoberg and Thomas Rothvoss.
\newblock A fourier-analytic approach for the discrepancy of random set
  systems.
\newblock {\em CoRR}, abs/1806.04484, 2018.

\bibitem[KLP12]{KLP12}
Greg Kuperberg, Shachar Lovett, and Ron Peled.
\newblock Probabilistic existence of rigid combinatorial structures.
\newblock In {\em Proceedings of the 44th Symposium on Theory of Computing
  Conference, {STOC} 2012, New York, NY, USA, May 19 - 22, 2012}, pages
  1091--1106, 2012.

\bibitem[LM15]{LM15}
Shachar Lovett and Raghu Meka.
\newblock Constructive discrepancy minimization by walking on the edges.
\newblock {\em {SIAM} J. Comput.}, 44(5):1573--1582, 2015.

\bibitem[Mat95]{MAT95}
Ji{\v{r}}{\'{\i}} Matou{\v{s}}ek.
\newblock Tight upper bounds for the discrepancy of half-spaces.
\newblock {\em Discrete {\&} Computational Geometry}, 13:593--601, 1995.

\bibitem[Rot17]{ROT17}
Thomas Rothvoss.
\newblock Constructive discrepancy minimization for convex sets.
\newblock {\em {SIAM} J. Comput.}, 46(1):224--234, 2017.

\bibitem[Spe85]{SPE85}
Joel Spencer.
\newblock Six standard deviations suffice.
\newblock {\em Transactions of the American Mathematical Society},
  289(2):679--706, 1985.

\bibitem[Spe88]{SPE88}
Joel Spencer.
\newblock Coloring the projective plane.
\newblock {\em Discrete Mathematics}, 73(1):213 -- 220, 1988.

\end{thebibliography}

\end{document}